\documentclass[a4paper]{amsart}
\usepackage{tikz,pgfplots}
\usepackage[utf8]{inputenc} 
\usepackage{
	amsmath,amssymb, 
	graphicx, 
	mathtools,
	bm,
	todonotes,
	hypbmsec,
	float,
	subcaption
}
\usepackage[shortlabels]{enumitem}
\usepackage[colorlinks=true]{hyperref}
\usepackage{etoolbox}
\usepackage[margin=3cm]{geometry}

\DeclarePairedDelimiterX\abs[1]\lvert\rvert{
	\ifblank{#1}{\:\cdot\:}{#1}
}
\DeclarePairedDelimiterX\norm[1]\lVert\rVert{ 
	\ifblank{#1}{\:\cdot\:}{#1}
}
\DeclarePairedDelimiterX{\inner}[2]{\langle}{\rangle}{ 
	\ifblank{#1}{\:\cdot\:}{#1},\ifblank{#2}{\:\cdot\:}{#2}
}

\providecommand\given{}
\DeclarePairedDelimiterX\Set[1]{\lbrace}{\rbrace}{
	\renewcommand\given{\SetSymbol[\delimsize]}
	#1
}
\DeclarePairedDelimiterX\condE[1]{\mathbb{E}\lbrack}{\rbrack}{
	\renewcommand\given{\SetSymbol[\delimsize]}
	#1
}
\DeclarePairedDelimiterX\condVar[1]{\text{Var}(}{)}{
	\renewcommand\given{\SetSymbol[\delimsize]}
	#1
}

\DeclarePairedDelimiterXPP\Prob[1]{\mathbb{P}}(){}{
	\renewcommand\given{\nonscript\:\delimsize\vert\nonscript\:
		\mathopen{}}
	#1}
\DeclarePairedDelimiterXPP\Var[1]{\text{Var}}(){}{
	\renewcommand\given{\nonscript\:\delimsize\vert\nonscript\:
		\mathopen{}}
	#1}
\DeclarePairedDelimiterXPP\Mean[1]{\mathbb{E}}[]{}{
	\renewcommand\given{\nonscript\:\delimsize\vert\nonscript\:
		\mathopen{}}
	#1}

\DeclarePairedDelimiterX\open[2]
	() 
	{#1,#2}
\DeclarePairedDelimiterX\lopen[2]
	(] 
	{#1,#2}
\DeclarePairedDelimiterX\ropen[2]
	[) 
	{#1,#2}
\DeclarePairedDelimiterX\closed[2]
	[] 
	{#1,#2}

\theoremstyle{plain}
\newtheorem{theorem}{Theorem}

\newtheorem{corollary}[theorem]{Corollary}
\newtheorem{proposition}[theorem]{Proposition}
\newtheorem{lemma}[theorem]{Lemma}

\newcommand\R{\mathbb{R}}

\usepackage[mathscr]{euscript} 

\DeclarePairedDelimiterX\lrangle[1]\langle\rangle{
	\ifblank{#1}{\:\cdot\:}{#1}
}

\newcommand\dd{\mathrm{d}} 
\newcommand\ff{\mathcal{F}} 
\newcommand\ee{\mathbb{E}} 

\newcommand\pp{\mathbb{P}}
\newcommand\convd{\overset{\sim}{\longrightarrow}} 

\newcommand{\n}{{(n)}}
\newcommand{\p}{\mathbb P}
\newcommand{\e}{\mathbb E}
\renewcommand{\a}{\alpha}
\newcommand{\stably}{\overset{st}{\rightarrow}}
\renewcommand{\convd}{\overset{d}{\rightarrow}}
\newcommand{\cip}{\overset{\p}{\rightarrow}}
\newcommand{\oh}{\mathrm{o}}

\newcommand{\ind}[1]{\mbox{\rm\large 1}_{\{#1\}}}

\title[Discretization of the Lamperti representation]{Discretization of the Lamperti representation of a positive self-similar Markov process}
\author[J.\ Ivanovs and J.\ D.\ Th{\o}stesen]{Jevgenijs Ivanovs and Jakob D.\ Thøstesen}

\begin{document}
\begin{abstract}
This paper considers discretization of the L\'evy process appearing in the Lamperti representation of a strictly positive self-similar Markov process.
Limit theorems for the resulting approximation are established under some regularity assumptions on the given L\'evy process.
Additionally, the scaling limit of a positive self-similar Markov process at small times is provided.
Finally, we present an application to simulation of self-similar L\'evy processes conditioned to stay positive.
\end{abstract}
\maketitle

\section{Introduction}
Positive self-similar Markov processes (pssMp) have received a lot of attention in recent years, see~\cite{bardoux16,pssMp_fluct} and a survey~\cite{rivero_survey}. One class of examples is given by self-similar L\'evy processes `conditioned' to stay positive, which arise in various limit theorems concerned with extremes, first passage times and Skorokhod reflection~\cite{asmussen_ivanovs, ivanovs2018, iva_podolskij}.
Recall that $X=(X_t)_{t\geq 0}$ is a pssMp if it is a positive strong Markov process with the self-similarity property:
$(X_{tc})_{t\geq 0}$ with $X_0=x>0$ has the law of $(c^{1/\alpha} X_t)_{t\geq 0}$ with $X_0=c^{-1/\a}x$ for any $c>0$, where $1/\alpha>0$ is sometimes called the Hurst index.
Throughout this work we restrict our attention to strictly positive~$X$.

The fundamental result of \cite{lamperti72} states that every pssMp $X$ (not hitting 0) can be represented via a L\'evy process $\xi$ as follows:
\begin{equation}\label{eq:lamperti}
  	X_t=x\exp(\xi_{\tau(tx^{-\alpha})}),\qquad \tau(r):=\inf\Set{s>0\given I_s\geq r},\qquad I_s:=\int_0^s \exp(\alpha \xi_u)\dd u
\end{equation}
where $\limsup_{t\to\infty}\xi_t=\infty$ a.s., and $x>0$ is a given starting position. Moreover, this relation can be inverted to obtain $\xi$ in terms of~$X$.
The Lamperti representation is key for deriving various properties of pssMp~\cite{pardo_envelope}.
Furthermore, it also provides a way to simulate from the law of~$X$, which is important in application of the above mentioned limit theory and beyond.

The purpose of this paper is to investigate the basic discretization scheme, where the path of the L\'evy process $\xi$ is sampled at equidistant times $i/n,i\in \mathbb N$.
Throughout this work we assume that the increment $\xi_{1/n}$ can be sampled exactly and efficiently.
This allows to approximate the integral function $I$, which is then used to construct an approximation $X^\n$ of $X$ at the times of interest.
Our main result is the limit theorem for the scaled error $a_n(X_{t}-X^\n_t)$ as $n\to \infty$, as well as its multidimensional version concerning a finite set of times, see Corollary~\ref{cor:mainFDD}.
This result crucially depends on the limit theory for the integrated process error in~\cite{jacodpaper}, which is extended to include zooming-in on $\xi$~\cite{ivanovs2018} at inverse times.

A result of independent interest is presented in Theorem~\ref{thm:ppsMpzooming},
which complements the classical law of the iterated logarithm for a pssMp at small times~\cite[Thm.\ 7.1]{lamperti72}.
We show that $a_n(X_{t/n}-x)_{t\geq 0}$ has a non-trivial  weak limit as $n\to\infty$ under the obvious regularity condition that there is weak convergence to a non-zero limit for some fixed $x,t>0$ and some positive function~$a_n$. Furthermore, this assumption is equivalent to the regularity of the underlying L\'evy process, which we assume in the above discussed approximation theory.

This work has been motivated by the problem of simulating a stable L\'evy processes conditioned to stay positive, see~\cite[\S 4]{engelke_ivanovs} for various available representations.
Importantly, the most obvious methods result in infinite expected running times. One of the reasons is that for an oscillating L\'evy process the first passage time over a fixed level has infinite expectation.
In this regard we note that~\cite{gonzalez2019varepsilon} recently provided an $\varepsilon$-strong simulation algorithm for the convex minorants of stable meanders, which are closely related to conditioned processes.

The structure of this paper is as follows. We start with the definitions, assumptions and necessary basic theory in \S\ref{sec:prereq}. The limit theory for the approximation is derived in \S\ref{sec:mainresults} relying on the joint stable convergence of some fundamental objects which is proven later in \S\ref{sec:proof}. The scaling limit of a pssMp is studied in \S\ref{sec:zooming} relying on a basic convergence result for L\'evy processes which may be of an independent interest. In \S\ref{sec:application} these results are applied to self-similar L\'evy processes conditioned to stay positive, where we also provide a numerical illustration in the simplest setting of a standard Brownian motion.
We conclude with \S\ref{sec:comments} providing comments about the density assumption and the trapezoidal approximation of the integral.

\section{Definitions and prerequisites}\label{sec:prereq}
\subsection{Fundamentals}
We work with càdlàg processes on a filtered probability space $(\Omega,\mathcal F,(\mathcal F_t)_{t\geq 0},\p)$ and use the Skorokhod $J_1$ topology.
Let $\xi=(\xi_t)_{t\geq 0}$ be a L\'evy process, that is, an adapted càdlàg process such that $\xi_{t+s}-\xi_t$ is independent of $\mathcal F_t$ and has the law of $\xi_s$ for any $t,s\geq 0$.
Furthermore, as indicated above we assume that 
\begin{equation}\label{eq:limsup}
\limsup_{t\to\infty}\xi_t=\infty\qquad \text{a.s.},
\end{equation}
which is satisfied, for example, if $\xi_1$ is integrable and $\e \xi_1\geq 0$, excluding the trivial 0 process. 

Our results are slightly cleaner when formulated using a certain L\'evy process $\widehat \xi$ on the real line. This is the same as saying that $(-\widehat\xi_{(-t)-})_{t\geq 0}$ is an independent copy of $(\widehat\xi_t)_{t\geq 0}$, 
where the left limit is needed to get a càdlàg path over the real line.
 In particular, 
\begin{equation}\label{eq:levyR}\widehat \xi_T\overset{d}{=}{\rm sign}(T)\widehat\xi_{|T|}\end{equation}
for any random $T\in\R$ independent of~$\widehat \xi$, because L\'evy processes do not jump at fixed times.

The concept of stable convergence~\cite{aldous,renyi} is fundamental in discretization of processes~\cite{jacodbook,podolskij2010}.
Consider a sequence of random variables $Z_n$ defined on $(\Omega,\ff,\p)$ and taking values in some Polish space. 
The sequence $Z_n$ is said to converge stably to $Z$ $(Z_n\stably Z)$ defined on an extension $(\widetilde{\Omega},\widetilde{\ff},\widetilde{\pp})$ if
 \begin{equation}\label{eq:stable_def}
    \ee[f(Z_n)Y]\to\widetilde{\ee}[f(Z)Y]
 \end{equation}
  	for all bounded continuous functions $f$ and all bounded $\ff$-measurable~$Y$, see also~\eqref{eq:stable_simple} below for further intuition.
The standard example concerns $Z$ being independent of $\mathcal F$-measurable $Y$, and then the term mixing convergence is sometimes used.

\subsection{Approximation}
Consider the discretized process $\xi^{(n)}$ given by $\xi^{(n)}_t=\xi_{[tn]/n}$, where $[x]$ denotes the integer part of $x$. Later we also use the fractional part $\{x\}=x-[x]$.
The basic approximation of the integrated process $I_t$ in~\eqref{eq:lamperti} is given by the left Riemann sum 
\begin{equation*}
  	I^{(n)}_t:=\int_0^t \exp(\alpha\xi^{(n)}_s)\dd s=\frac{1}{n}\sum_{k=1}^{[tn]}\exp(\alpha\xi_{(k-1)/n})+\frac{\{tn\}}{n}\exp(\alpha \xi_{[tn]/n}).
\end{equation*}
In \S\ref{sec:trapezoid} below we also comment on the use of the trapezoid rule.

Note that the integrals $I_t$ and $I^\n_t$ are continuous and strictly increasing from 0 to $\infty$ a.s., which is an easy consequence of~\eqref{eq:limsup}.
Since $\xi$ has countably many jumps, we see that $I^\n$ converges to $I$ pointwise a.s. 
Define the respective inverse
 \[\tau_n(r):=\inf\Set{s>0\given I^\n_s\geq r},\qquad r\geq 0,\]
 and observe that a.s.\ both $\tau(r)$ and $\tau_n(r)$ are finite and 
\[\tau_n(r)\to \tau(r).\]
Finally, we use the approximation
\[X^\n_t:=x\exp(\xi^\n_{\tau_n(t x^{-\alpha})})=x\exp(\xi_{[\tau_n(t x^{-\alpha})n]/n}),\]
since $\xi$ is sampled at $k/n$ only.

Let us note that $X_t^\n\to X_t$ a.s., because of the continuity of $\xi$ at $\tau(r)$.
The latter readily follows from quasi left-continuity of $\xi$~\cite[Prop.\ I.7]{bertoinbook} and the fact that $I_t$ is continuous and strictly increasing. Hence the main question concerns the speed of convergence.
Finally, observe that $X^\n_t$ coincides with $X_{T^\n}$ for some $T^\n\to t$ a.s., that is, sampling is exact up to time perturbation. More precisely, such $T^\n$ is given by
\begin{equation}\label{eq:time_perturb}
T^\n=x^{\a}I_{[\tau_n(t x^{-\alpha})n]/n},
\end{equation}
so that $\tau(T^\n x^{-\a})=[\tau_n(t x^{-\alpha})n]/n$. The corresponding limit result is also given in the following.
 
Figure \ref{subfig:discretization} below illustrates the discretization of $\xi$ in the case where $n=10$, $\alpha=2$ and $\xi$ is a Brownian motion with unit variance and drift $1/2$ (corresponding to $X$ being a standard Brownian motion conditioned to stay positive, see \S\ref{sec:application}). In Figure \ref{subfig:discretizedintegral} we see the integrals $I$ and $I^\n$ and their inverses at $r=1$.

\begin{figure}[h]
	\centering
	\begin{subfigure}[t]{.48\textwidth}
		\vskip 0pt
		\centering
		\begin{tikzpicture}
			\begin{axis}[tick pos=left, width=1.15\linewidth, ymin=-1.8,ymax=18,ytick={0,5,10,15},xmin=-0.11,xmax=1.1,xtick={0,0.2,0.4,0.6,0.8,1}]
				\addplot+ [mark=none, color=blue] table[x="V1",y="V2", col sep=comma] {exponentiallevy.csv};
				\addplot+ [jump mark left,mark=none,mark options={fill=red, scale=.5}, color=red] table[x="V1",y="V2", col sep=comma] {discretizedexponentiallevy.csv};
			\end{axis}
		\end{tikzpicture}
		\captionsetup{width=.8\linewidth}
		\caption{The processes $\exp(\a\xi_t)$ in blue and $\exp(\a\xi^\n_t)$ in red with $n=10$.}
		\label{subfig:discretization}
	\end{subfigure}\quad%
	\begin{subfigure}[t]{.48\textwidth}
		\vskip 0pt
		\centering
		\begin{tikzpicture}
			\begin{axis}[tick pos=left, width=1.15\linewidth,ymin=-0.53,ymax=5.3,ytick={0,1,2,3,4,5},xmin=-0.11,xmax=1.1,xtick={0,0.2,0.4,0.6,0.8,1}]
				\addplot+ [mark=none, color=blue, solid] table[x="V1",y="V2", col sep=comma] {integral.csv};
				\addplot+ [mark=none, color=red, solid] table[x="V1",y="V2", col sep=comma] {discretizedintegral.csv};
				\draw[dashed,>=stealth,color=blue](axis cs:0.5876691,-0.53)--(axis cs:0.5876691,1);
        		\draw[dashed,>=stealth,color=red](axis cs:0.6474492,-0.53)--(axis cs:0.6474492,1);
        		\draw[dashed,>=stealth,color=black](axis cs:-0.11,1)--(axis cs:1.1,1);
			\end{axis}
		\end{tikzpicture}
		\captionsetup{width=.8\linewidth}
		\caption{Integrals of the processes in Figure \ref{subfig:discretization}. The colored dashed lines mark $\tau(1)$ and $\tau_n(1)$.}
		\label{subfig:discretizedintegral}
	\end{subfigure}
	\caption{An illustration of the discretization and how it effects calculation of $\tau(1)$.}
	\label{fig:discretizationplots}
\end{figure}
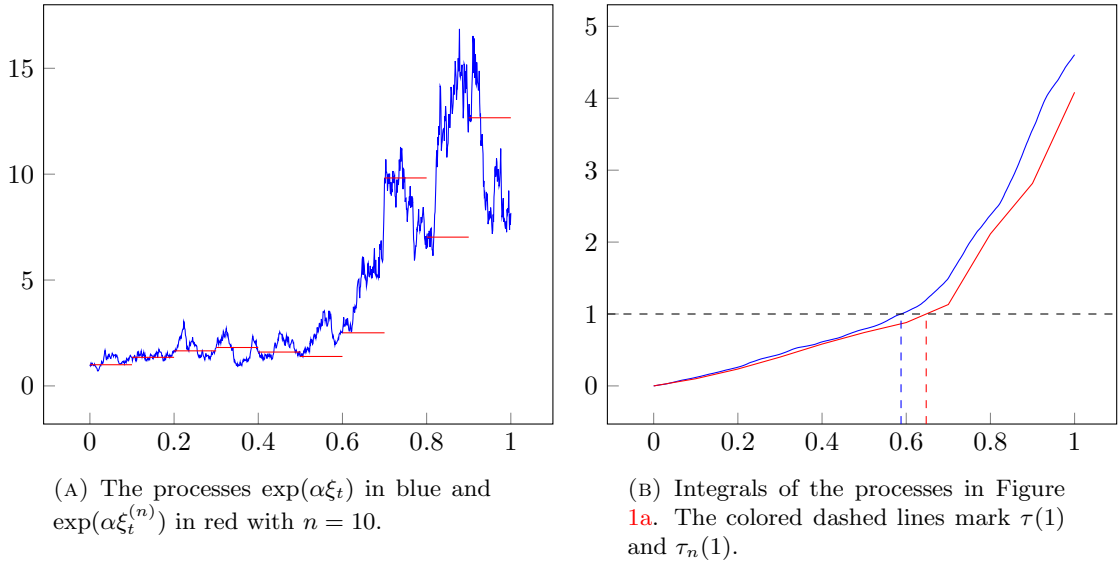

\subsection{Integrated process error}\label{sec:integrated}
Integrated discretization error for It\^o semimartingales has been studied in~\cite{jacodpaper}, see also~\cite[Ch.\ 6]{jacodbook}. 
In our case the function of interest is $f(x)=\exp(\alpha x)$.
Let us first describe the limiting process defined on an extension of the original probability space:
\begin{align}\label{eq:Delta}
  	\Delta_t=\frac{\sigma^2}{\sqrt{12}}\int_0^t f'(\xi_s)\dd W'_s+\sum_{m:T_m\leq t}(f(\xi_{T_m})-f(\xi_{T_m-}))(\kappa_m-\frac{1}{2})+\frac{1}{2}(f(\xi_t)-f(0)).
\end{align}
Here  $W'$ is a standard Brownian motion and $(\kappa_m)_{m\geq 1}$ is an i.i.d.\ sequence of standard uniforms, independent of each other and of~$\ff$.
Furthermore, $(T_m)_{m\geq1}$ denotes a weakly exhausting sequence of the jump times of $\xi$, and $\sigma^2$ is the variance of the Brownian component of~$\xi$.
The filtered extension is taken to be 'very good'~\cite[\S 2.1.4]{jacodbook} so that, in particular, $\Delta_t$ is adapted to $\widetilde{\ff}_t$.

\begin{theorem}[Jacod, Jakubowski, \& M\'emin~\cite{jacodpaper}]\label{thm:jacod}
  	There is the convergence
  	\begin{equation}\label{eq:trapezoidapprox1}
    	(\Delta^\n_t)_{t\geq 0}:=n\left(I_{[tn]/n}-I^{(n)}_{[tn]/n}\right)_{t\geq 0}\stably(\Delta_t)_{t\geq 0},
  	\end{equation}
  	where $\Delta_t$ is defined in~\eqref{eq:Delta}.
\end{theorem}

The result is stated for the difference of the integral and its approximation up to the last epoch $[tn]/n$ rather than time $t$. In fact, there is no functional convergence in Skorokhod's $J_1$-topology in the latter setting unless $\xi$ is continuous, see also \cite{jacodpaper}.
    	The problem here is that the jumps enter into the limit expression, whereas the pre-limit evolves continuously approximating these jumps by steep (almost linear) curves. Intuitively, this can be remedied by switching to Skorokhod's weaker $M_1$-topology, where the completed graphs of paths are compared.
    	We do not pursue this question in the present paper, though.

\subsection{Regularity of the L\'evy process}
Not surprisingly, our limit result requires certain regularity of the process~$\xi$.
Following~\cite{ivanovs2018} we assume that there exists a positive scaling function $a_n>0$ and a random variable $\widehat\xi_1\neq 0$ such that
\begin{equation}\label{eq:zoominginlimit}\tag{A1}
    		a_n\xi_{\frac{1}{n}}\convd\widehat\xi_1, \qquad\text{ as }\mathbb R\ni n\to\infty.
  		\end{equation}
Necessarily, $\widehat\xi_1$ is infinitely divisible and the corresponding L\'evy process $\widehat\xi$ is $1/\beta$-self-similar with $\beta\in(0,2]$, whereas $a_n$ is regularly varying at infinity with index $1/\beta$. Importantly, this convergence extends to the weak convergence for processes: 
\begin{equation}\label{eq:levyzoomingin}
\left(a_n\xi_{t/n}\right)_{t\geq 0}\convd(\widehat{\xi}_t)_{t\geq 0}.
\end{equation}
Intuitively, this is understood as zooming in on $\xi$ at the origin.
It must be noted that~\eqref{eq:zoominginlimit} can be formulated in terms of the L\'evy triplet of~$\xi$,
yielding the parameters of $\widehat \xi$ and the scaling function~$a_n$~\cite[Thm.\ 2]{ivanovs2018}. See also~\cite{bisewski_iva} for further examples and simple sufficient conditions.
Finally, it will be shown later that the convergence in~\eqref{eq:levyzoomingin} is, in fact, stable and the limiting $\widehat \xi$ is independent of~$\ff$.

Our limit theory will also require convergence of $\{\tau(r)n\}$.
The classic result of~\cite{K37} states that such a sequence converges to a standard uniform random variable if the distribution of $\tau(r)$ is absolutely continuous, see also~\cite{tukey} for sufficient and necessary conditions.
Again the convergence is stable and the limiting uniform is independent of~$\tau(r)$, see~\cite{jacodbook}. 
We impose a slightly stronger assumption:
\begin{equation}\label{eq:density}\tag{A2}
\text{The law of }(\tau(r),\xi_{\tau(r)})\text{ is absolutely continuous for every (small) }r>0.
\end{equation}
This assumption on the inverse can be replaced by an assumption on the integral~$I$.
More precisely, in \S\ref{sec:density} we show that it is sufficient to assume that the pair
\[\big(\int_0^t \exp(\a \xi_s)\dd s,\exp(\a \xi_t)\big)\]
has a density $g_t(x,y)$ which is jointly continuous in $t,x,y>0$.
The latter question concerns the exponential functional and has been studied in a number of papers, see~\cite{vostrikova, cpy, PRvS}.
Verification of this condition, however, is still non-trivial and thus we avoid assuming~\eqref{eq:density} in various places, including \S\ref{sec:rate} which establishes the rate of convergence of our approximation.

\section{Approximation results}\label{sec:mainresults}
Throughout this paper we assume~\eqref{eq:limsup}. The assumptions \eqref{eq:zoominginlimit} and~\eqref{eq:density} are needed only for some results, and this is stated at the corresponding places.

Our main aim is to establish a limit result (as $n\to\infty$) for the scaled relative error, which according to~\eqref{eq:lamperti} is given by
\begin{equation}\label{eq:step1}a_n\frac{X_t-X^\n_t}{X_t}=a_n(\xi_{\tau(r)}-\xi_{[\tau_n(r)n]/n})(1+\oh_\p(1)),\qquad r=tx^{-\a},\end{equation}
where we also use the mean value theorem and the fact that $\xi$ is continuous at~$\tau(r)$. The scaling sequence $a_n>0$ will be chosen according to~\eqref{eq:zoominginlimit} in the following.
Letting \[\widehat\xi^\n=a_n(\xi_{\tau(r)+t/n}-\xi_{\tau(r)})_{t\in\R}\] be the two-sided process arising upon zooming in on $\xi$ at $(\tau(r),\xi_{\tau(r)})$, we find that
\begin{equation}\label{eq:step2}a_n(\xi_{\tau(r)}-\xi_{[\tau_n(r)n]/n})=-\widehat\xi^\n_{[\tau_n(r)n]-\tau(r)n}.
\end{equation}
Hence we need to establish the joint limit of the two-sided process $\widehat\xi^\n$ and the scaled time difference $\tau(r)n-[\tau_n(r)n]$, and to further extend it to the multivariate setting with $0<t_1<\cdots<t_d$.
It will be shown that the scaled time differences $n(\tau(r_i)-\tau_n(r_i))$ are not affected by infinitesimally small time intervals, whereas the zoomed-in processes are given by the local behavior of $\xi$ at $\tau(r_i)$ 
and in the limit result in independent copies of~$\widehat\xi$.

\subsection{Time variable and the inverse}
Our first result concerns the error in the inverse~$\tau_n(r)$. The limiting random variable $L(r)$, defined below, will play an important role in the following. 
\begin{proposition}\label{prop:tau}
  	For any $r>0$ it holds that
  	\begin{equation}\label{eq:L}
    	n(\tau(r)-\tau_n(r))\stably L(r):=-\Delta_{\tau(r)}\exp(-\alpha \xi_{\tau(r)}),
  	\end{equation}
  	where the process $\Delta_t$ is defined in~\eqref{eq:Delta}. 
\end{proposition}
\begin{proof}
First, we show that
\begin{equation}\label{eq:integralapproxattau1}
    	n(I_{\tau(r)}-I^\n_{\tau(r)})\stably\Delta_{\tau(r)}.
  	\end{equation}
  	Recall that $\xi$ is continuous at $\tau(r)$ a.s., and note that the same is true for the $\Delta$ process.
  	Hence by Theorem~\ref{thm:jacod} and the continuous mapping theorem we have the stated convergence, where $\tau(r)$ is replaced by $t_n(r):=[\tau(r)n]/n$.
It is left to show that the remaining term vanishes. This term, is upper bounded by
\[\exp(\a \sup\Set{\xi_t\given t\in[t_n(r),\tau(r)]})-\exp(\a\xi_{t_n(r)})\]
converging to 0 a.s.\ by the continuity of $\xi$ at $\tau(r)$. The lower bound is treated analogously.

Next, observe that
\[n\int_{\tau(r)}^{\tau_n(r)}\exp(\a\xi^\n_t)\dd t=n(r-I^\n_{\tau(r)})=n(I_{\tau(r)}-I^\n_{\tau(r)})\stably \Delta_{\tau(r)}.\]
Similar bounds to above show that the left-hand side is given by \[n(\tau_n(r)-\tau(r))\exp(\a\xi_{\tau(r)})\] times a term converging to 1 a.s.
The result readily follows.
\end{proof}
Observe that Proposition \ref{prop:tau} above easily extends to a multivariate version with $0<r_1<\cdots<r_d$.

\subsection{Rate of convergence}\label{sec:rate}
In order to proceed we need to supplement the convergence in Theorem~\ref{thm:jacod} by zooming in on $\xi$ at the times $\tau(r_i)$.
\begin{theorem}\label{thm:joint}
Assume \eqref{eq:zoominginlimit}. For any $0<r_1<\cdots<r_d$ and $r_i^n\to r_i$ there is the stable convergence
\[\Big((\Delta^\n_t)_{t\geq 0},\big(a_n(\xi_{\tau(r^n_i)+t/n}-\xi_{\tau(r^n_i)})_{t\in\R}\big)_{i=1,\ldots, d}\Big)\stably\Big((\Delta_t)_{t\geq 0},\big((\widehat \xi^i)_{t\in\R}\big)_{i=1,\ldots,d}\Big),\]
where $\widehat\xi^i$ are independent copies of $\widehat\xi$, also independent of everything else.
\end{theorem}
The proof of this result is postponed to~\S\ref{sec:proof1}.
It is very important that the time $t$ is allowed to be negative, which is a non-trivial extension of the case $t\geq 0$. This is needed, because the discretized epoch $[\tau_n(r)n]/n$ may be smaller than $\tau(r)$. 
Now the arguments underlying Proposition~\ref{prop:tau} readily yield the joint stable convergence:
\begin{equation}\label{eq:jointtau}\Big(n(\tau(r_i)-\tau_n(r_i)),a_n(\xi_{\tau(r_i)+t/n}-\xi_{\tau(r_i)})_{t\in\R}\Big)_{i=1,\ldots, d}\stably\Big(L(r_i),(\widehat \xi^i)_{t\in\R}\Big)_{i=1,\ldots,d}.\end{equation}

Next, we turn our attention to the pssMp and reconsider~\eqref{eq:step1} and~\eqref{eq:step2}.
Note that~\eqref{eq:jointtau} readily yields the result for the error in approximation of $X$ where $\tau_n(r)$ is used instead of $[\tau_n(r)n]/n$, but we do not observe $\xi_{\tau_n(r)}$. 
Our main limit theorem presented in \S\ref{sec:limit} requires further work and assumptions, whereas here we establish the rate of convergence in our pssMp approximation up to a bounded stochastic term.

Consider~\eqref{eq:step1} and the respective upper bound:
\[a_n(\xi_{\tau(r)}-\xi_{[\tau_n(r)n]/n})\leq a_n(\xi_{\tau(r)}-\inf_{t\in [0,1]}\xi_{\tau_n(r)-t/n})=:\overline B^\n(r),\]
and analogous lower bound $\underline B^\n(r)$ when using~$\sup$. According to~\eqref{eq:jointtau} we have
\[\overline B^\n(r)=-\inf_{t\in[0,1]}\widehat\xi^\n_{-n(\tau(r)-\tau_n(r))-t}\stably -\inf_{t\in[0,1]}\widehat\xi_{-L(r)-t},\]
because $\widehat\xi$ does not jump at fixed times. Since $(-\widehat\xi_{(-t)-})_{t\in\R}$ has the same law as $(\widehat\xi_t)_{t\in\R}$, we conclude that 
\begin{equation}\label{eq:bounds}\Big(\overline B^\n(r_i),\underline B^\n(r_i)\Big)_{i=1,\ldots,d}\stably \Big(\sup_{t\in[0,1]}\widehat\xi_{L(r_i)+t},\inf_{t\in[0,1]}\widehat\xi_{L(r_i)+t}\Big)_{i=1,\ldots,d}.\end{equation}
The following result is now immediate from~\eqref{eq:step1}. 
It establishes the rate of convergence $a_n^{-1}$ and provides explicit limiting bounds.
\begin{corollary}\label{cor:bonds}
Assuming~\eqref{eq:zoominginlimit}, for any $x>0$ and $0<t_1<\cdots<t_d$ it holds that
\[\underline B^\n(t_ix^{-\a})-\oh_\p(1)\leq a_n\left(\frac{X_{t_i}-X_{t_i}^\n}{X_{t_i}}\right)\leq \overline B^\n(t_ix^{-\a})+\oh_\p(1),\quad i=1,\ldots,d,\]
where the joint limit for the bounds is given in~\eqref{eq:bounds}.
\end{corollary}

\subsection{Discretization error in pssMp}\label{sec:limit}
More precise analysis requires further work and it hinges on the assumption~\eqref{eq:density} implying, in particular, that $\{\tau(r)n\}$ converges to the standard uniform distribution. 
We have the following generalization of Theorem~\ref{thm:joint}.
\begin{theorem}\label{thm:jointU} Consider $0<r_1<\cdots<r_d$ and assume~\eqref{eq:zoominginlimit} and~\eqref{eq:density}. Then
\begin{align*}\Big((\Delta^\n_t)_{t\geq 0},\big(\{\tau(r_i)n\},a_n(\xi_{\tau(r_i)+t/n}-\xi_{\tau(r_i)})_{t\in\R}\big)_{i=1,\ldots, d}\Big)
\stably\Big((\Delta_t)_{t\geq 0},\big(U_i,(\widehat\xi^i_t)_{t\in\R}\big)_{i=1,\ldots, d}\Big),\end{align*}
where $U_i$ are independent standard uniforms, also independent of the rest.
\end{theorem}
The proof of this result is given in~\S\ref{sec:proof} below. This readily yields an extension of \eqref{eq:jointtau} including the variables $\{\tau(r_i)n\}$ and their uniform limits.



\begin{corollary}\label{cor:mainFDD}
Assume \eqref{eq:zoominginlimit} and~\eqref{eq:density}. 
Then for any $x>0$ and $0<t_1<\cdots<t_d$ we have
\[\left(a_n\frac{X_{t_i}-X_{t_i}^\n}{X_{t_i}}\right)_{i=1,\ldots,d}\stably \left(\widehat{\xi^i}_{L(t_ix^{-\a})+U_i}\right)_{i=1,\ldots,d},\]
where $L(\cdot)$ is defined in  \eqref{eq:L}.
\end{corollary}
\begin{proof}
Using the identity
\[a-[b]=\{a\}-[\{a\}-(a-b)]\]
we observe that
\begin{equation}\label{eq:L+U} n(\tau(r_i)-[\tau_n(r_i)n]/n)=\tau(r_i)n-[\tau_n(r_i)n]\stably U_i-[U_i-L(r_i)]=:L(r_i)+U'_i,\end{equation}
because $U_i-L(r_i)$ has no mass at integers. 
It is easy to verify that $U_i'$ are again standard uniforms independent of $L(r_i)$ and the rest (excluding the respective $U_i$). 
Furthermore, jointly with the above we also have the zooming-in limits $\widehat\xi^i$, and so the representations~\eqref{eq:step1} and~\eqref{eq:step2} yield the limit $(-\widehat\xi^i_{-L(r_i)-U_i'})_{i=1,\ldots,d}$.
Now the result follows.
\end{proof}
It is noted that the limiting vector has dependent components and its realization depends on the realization of $\xi$ via~$L$.
Recall that $\widehat\xi^i$ is $1/\beta$-self-similar, which together with~\eqref{eq:levyR} and its independence of the rest yields an alternative representation of the limit components in Corollary~\ref{cor:mainFDD}:
\[{\rm sign}(L(t_ix^{-\a})+U_i)|L(t_ix^{-\a})+U_i|^{1/\beta}\widehat \xi^i_1.\]

Finally, we also have the limit result for the time shift defined in~\eqref{eq:time_perturb}:
\[n(t-T^\n)=x^{\a}n(I_{\tau(tx^{-\a})}-I_{[\tau_n(t x^{-\alpha})n]/n})\stably (L(tx^{-\a})+U)X_t^\a,\] 
by means of~\eqref{eq:L+U}. That is, our procedure yields the samples of $X_t$ up to a time shift of order $n^{-1}$.

\section{Proof of the joint convergence}\label{sec:proof}
Reconsider the definition of stable convergence in~\eqref{eq:stable_def}.
In this paper $Z_n$ is derived from the L\'evy process~$\xi$, and the limit $Z$ is constructed from $\xi$ and some additional random variables independent of~$\ff$.
Thus it is sufficient to take $\sigma(\xi)$-measurable $Y$ in~\eqref{eq:stable_def} to ensure the stable convergence, see also~\cite[p.\ 110]{jacodbook}.
Furthermore, it is sufficient to show
\begin{equation}\label{eq:stable_simple}
(Z_n,\xi_{t_1},\ldots,\xi_{t_k})\convd (Z,\xi_{t_1},\ldots,\xi_{t_k})
\end{equation}
for an arbitrary finite set of times $t_1,\ldots,t_k>0$.  
\subsection{Reinforcement of convergence results}\label{sec:proof1}
This subsection consists of sequential reinforcement of convergence results stated in~\eqref{eq:levyzoomingin} and in Theorem~\ref{thm:jacod}, and culminates with the proof of Theorem~\ref{thm:joint}.
\begin{lemma}\label{lem:stable_taun}
Assume~\eqref{eq:zoominginlimit} and let $\tau_n$ be a sequence of finite stopping times.
Then \[(\widehat\xi^\n_t)_{t\geq 0}:=a_n\Big(\xi_{\tau_n+t/n}-\xi_{\tau_n}\Big)_{t\geq 0}\stably \Big(\widehat\xi_t\Big)_{t\geq 0},\]
where $\widehat\xi$ is independent of~$\ff$. 
\end{lemma}
\begin{proof}
It is sufficient to consider the process $\widehat\xi^\n$ on some time interval $[0,T]$ jointly with $\xi$ at some times $t_1<\cdots<t_k$, see~\eqref{eq:stable_simple}. 
That is, we need to show that
\[\Big((\widehat \xi^\n_t)_{t\in[0,T]},(\xi_{t_i})_{i=1,\ldots,k}\Big)\convd \Big((\widehat \xi_t)_{t\in[0,T]},(\xi_{t_i})_{i=1,\ldots,k}\Big)\]
with an independent $\widehat \xi_t$.
Write $\xi_{t_i}=X^\n_i+Y^\n_i$, where $Y^\n_i$ are independent of $\widehat \xi^\n_t,t\in[0,T]$ and $X^\n_i\cip 0$, which can be achieved by considering independent increments over time intervals separated by $\tau_n$, $\tau_n+T/n$ and $t_i$.
Now, we may ignore $X^\n_i$, but then the stated convergence is immediate from independence and convergence of marginals. Here we also note that the limit process $\widehat\xi$ does not jump at~$T$ and hence the restrictions converge.
\end{proof}


\begin{lemma}\label{lem:jointzooming}
Assume~\eqref{eq:zoominginlimit} and let $\tau_n$ be a sequence of finite stopping times. It holds as $n\to\infty$ that
\[\Big(\Delta^{\n},\widehat\xi^\n\Big)\stably\Big(\Delta,\widehat\xi\Big),\]
where $\widehat\xi$ is independent of everything else.

Moreover, if $0\leq \tau_n^1<\cdots<\tau_n^d<\infty$ are stopping times for each $n$ and such that $n(\tau_n^{i+1}-\tau_n^{i})\to\infty$ a.s. for all $i=1,\ldots, d-1$ then the multivariate version holds with the corresponding limits $\widehat\xi^i$ being independent copies of $\widehat \xi$, also independent of everything else. 
\end{lemma}
\begin{proof}
Again we may restrict the processes $\widehat\xi^\n$ to some time interval $[0,T]$.
Let $\overline \tau_n$ be the discretization epoch right after $\tau_n+T/n$, and note that $\overline \tau_n$ is a stopping time independent of $\widehat\xi^\n$.
The idea is to replace $\Delta^{\n}$ with the integrated difference $\tilde\Delta^{\n}$, where the interval $[\tau_n,\overline \tau_n]$ and the respective space increment are ignored.
More precisely, the new $\xi$ is kept constant on $[\tau_n,\overline \tau_n]$ and then it has the original increments.
Observe that $\sup_{t\leq T'}|\tilde\Delta^{\n}_t-\Delta^{\n}_t|=\oh_\p(1)$ using the strong Markov property at $\tau_n$, see also the proof of Proposition~\ref{prop:tau}.
But now the two parts are independent and the arguments from Lemma~\ref{lem:stable_taun} can be repeated, additionally using Theorem~\ref{thm:jacod} for the joint convergence of $\Delta^{\n}$ and $\xi_{t_i}$.

The multivariate version follows the same reasoning. Note, that the intervals $[\tau_n^i,\tau_n^i+T/n]$ do not intersect with probability tending to~1 by assumption. Hence we may assume this property which then yields independent $\widehat\xi^i$.
\end{proof}

\begin{proof}[Proof of Theorem~\ref{thm:joint}]
We use Lemma~\ref{lem:jointzooming}  with $\tau^i_n=\tau(r_i^n-T/n)$ for a fixed $T>0$. Note that
\[n(\tau(r^n_i)-\tau(r^n_i-T/n))\to T\exp(-\a \xi_{\tau(r_i)})=:s_i\quad \text{a.s.},\] 
see also the proof of Proposition~\ref{prop:tau}.
Thus we can add the required time shifts to the limit result, and these limiting shifts $s_i$ are independent of the processes $(\widehat\xi_t^i)_{t\geq 0}$.
But for any $T'>0$ we can choose $T$ large enough so that with arbitrarily large probability $s_i>T'$,
and on this event $(\widehat\xi^i_{s_i+t}-\widehat\xi^i_{s_i})_{t\geq -T'}$ has the law of $(\widehat\xi_t)_{t\geq -T'}$ and is independent of~$s_i$.
\end{proof}
In conclusion, the stopping time $\tau(r)$ has a particular structure allowing to extend zooming in at $\tau(r)$ also to the negative times.

\subsection{Fractional parts and the standard uniform}
Here we prove the joint convergence in Theorem~\ref{thm:jointU} for $d=1$. For the purpose of extending it from $d=1$ to $d\geq 1$ later we need to allow for perturbations in $r$. We state this result as a separate lemma.
\begin{lemma}\label{lem:conv}
Assuming \eqref{eq:zoominginlimit} and~\eqref{eq:density} we have for any $r_n\to r>0$:
\[\Big(\Delta^\n,a_n(\xi_{\tau(r_n)+t/n}-\xi_{\tau(r_n)})_{t\in\R},\{\tau(r_n)n\}\Big)\stably\Big(\Delta,(\widehat \xi_t)_{t\in\R},U\Big),\]
where $U$ is  a standard uniform independent of everything else.
\end{lemma}

The independent uniform will arise via the following lemma.
Consider a random variable $Z$ and a sequence of random variables $(U_n,V_n,Y_n)$ defined on the same probability space. 
Let $\pp_z$ be the regular conditional distribution $\Prob{\cdot\given Z=z}$.
\begin{lemma}\label{lemma:XYUZstablelemma}
	Assume that $Y_n\convd Y$, and that for each $z$ (in the support of $Z$) we have under $\pp_z$:
	\begin{itemize}
		\item $U_n$ is independent of $(V_n,Y_n)$ for each $n$,
		\item The distribution of $U_n$ has no atoms and converges weakly to the standard uniform distribution.
	\end{itemize}
	Then $(Y_n,\{U_n+V_n\})\convd (Y,U)$ with a standard uniform $U$ independent of $Y$.
\end{lemma}
\begin{proof}
	Let $F_{n,z}$ be the continuous distribution function of $U_n|Z=z$. Define \[U'_n=F_{n,Z}(U_n)\] and note that, given $Z=z$, $U'_n$ is a standard uniform independent of $(V_n,Y_n)$. 
	Note that
	\[\p(Y_n\in B,\{V_n+U'_n\}\leq u)=u\p(Y_n\in B)\]
	by conditioning on $Z,Y_n,V_n$ and the fact that $\{v+U\}$ is standard uniform. Hence 
		\[(Y_n,\{V_n+U'_n\})\convd (Y,U)\] with $Y$ and $U$ independent.
		
	It is left to show that 
	\begin{equation}\label{eq:diff}\{V_n+U'_n\}-\{V_n+U_n\}\cip 0.\end{equation} 	
	Since $F_{n,z}(x)-x\to 0$ and the convergence is necessarily uniform in $x$, we see that $U_n'-U_n\cip 0$.
	Hence for any small $\delta>0$ we have $|U_n'-U_n|<\delta$ with probability at least $1-\delta$ for large~$n$.
	Moreover, $\p(\{V_n+U'_n\}<\delta)=\p(\{V_n+U'_n\}>1-\delta)=\delta$ and thus ~\eqref{eq:diff} does not exceed $\delta$ in absolute value with probability at least $1-3\delta$.
\end{proof}

\begin{proof}[Proof of Lemma~\ref{lem:conv}]
Choose $0<\delta<r$ and $\delta'>0$, and consider the process $\xi'_t=\xi_{\tau(r-\delta)+t}-\xi_{\tau(r-\delta)}$ independent of $\ff_{\tau(r-\delta)}$. Let $\tau$ be the time such that \[\int_0^\tau\exp(\a \xi'_t)\dd t=\delta'.\]
We consider the regular conditional distribution $\p_z$ corresponding to conditioning on $\xi'_\tau=z$.
Note that for every $z$ the variable $\tau$ has a density under $\p_z$ according to the assumption~\eqref{eq:density}, and so the distribution of $U_n=\{\tau n\}$ has no atoms and it converges weakly to a standard uniform law. 
Furthermore, 
\[(I_{\tau(r-\delta)+\tau},\xi_{\tau(r-\delta)+\tau})=(r-\delta+\exp(\a\xi_{\tau(r-\delta)})\delta',\xi_{\tau(r-\delta)}+z),\]
and we assume that the first component is smaller than $r_n\wedge(r-\delta/2)$; we may do so since this is true for small enough $\delta'$ with arbitrarily high probability.
Now letting $R_n=\tau(r_n)-(\tau(r-\delta)+\tau)$ be the remaining time, we note the decomposition of the fractional part of interest:
\[\{\tau(r_n)n\}=\{(\tau(r-\delta)+R_n)n+\{\tau n\}\}=:\{V_n+U_n\},\]
where $U_n$ is independent of $V_n$ under $\p_z$.

Next, we define the quantities of interest, which will be assembled into $Y_n$.
The integrated difference process stopped at $\tau(r-\delta)$ is denoted by $\hat\Delta^{\n}$. We consider this quantity jointly with $\xi_{t_i}\ind{t_i<\tau(r-\delta)}$ for some fixed times $t_i,i\leq k$.
Furthermore, consider the epoch $\tau_n$ following $\tau(r+\delta)$ with the corresponding incremental process $\tilde\xi_t=\xi_{\tau_n+t}-\xi_{\tau_n}$, which is independent of $\ff_{\tau(r+\delta)}$.
The integrated difference process for the times $\tau_n+t,t\geq 0$  is given by $\exp(\a\xi_{\tau(r+\delta)})(1+\oh_\p(1))\tilde\Delta^{\n}$, which is our second object of interest.
It is considered jointly with $\xi_{\tau(r+\delta)+\tilde t_i}=\xi_{\tau(r+\delta)}+\tilde \xi_{\tilde t_i}+\oh_\p(1)$ for some fixed $\tilde t_i,i\leq \tilde k$.
Thirdly, we consider the zoomed-in process $\widehat\xi^\n_t=a_n(\xi_{\tau(r_n)+t/n}-\xi_{\tau(r_n)})$ for $t\in[-T,T]$.
The event where $\tau(r+\delta)>\tau(r_n)+T/n$ and $\tau(r-\delta/2)<\tau(r_n)-T/n$ occurs with arbitrarily high probability, and we assume these inequalities in the following.
The above objects form the random quantity 
\[Y_n=\Big(\hat\Delta^{\n},(\xi_{t_i}\ind{t_i<\tau(r-\delta)})_{i=1,\ldots,k},\xi_{\tau(r+\delta)},\tilde\Delta^{\n},(\tilde \xi_{\tilde t_i})_{i=1,\ldots,\tilde k},\widehat\xi^\n\Big),\]
 and as required in Lemma~\ref{lemma:XYUZstablelemma} the variable $U_n$ is independent of $(V_n,Y_n)$ and the above events under $\p_z$.

Observe that the quantities $\tilde\Delta^\n,\tilde\xi_{\tilde t_i}$ are independent of the rest and have a joint weak limit as given by Theorem~\ref{thm:jacod}.
But the rest converges according to Theorem~\ref{thm:joint}, where stopping at $\tau(r-\delta)$ requires that $\xi$ does not jump at this time, which is indeed true. 
Thus Lemma~\ref{lemma:XYUZstablelemma} yields $(Y_n,\{\tau(r)n\})\convd (Y,U)$ with an independent standard uniform $U$ and obvious $Y$ for any given $\delta>0$. 

Finally, we piece together different components to get the integrated difference processes with the time interval $(\tau(r-\delta),\tau(r+\delta))$ excluded, as well as the corresponding limiting expression, see also~\eqref{eq:Delta}.
Now we can take $\delta\downarrow 0$ using \cite[Prop.\ 2.2.4]{jacodbook} to get
\begin{align*}
&\Big(\Delta^{\n},(\xi_{t_i}\ind{t_i<\tau(r)})_{i=1,\ldots,k},(\xi_{\tau(r)+\tilde t_i})_{i=1,\ldots,k},\widehat\xi^\n,\{\tau(r_n n)\}\Big)\convd\\
&\qquad\Big(\Delta,(\xi_{t_i}\ind{t_i<\tau(r)})_{i=1,\ldots,k},(\xi_{\tau(r)+\tilde t_i})_{i=1,\ldots,k},\widehat\xi,U\Big).
\end{align*}
It is only required to verify the assumptions of~\cite[Prop.\ 2.2.4]{jacodbook}.
Firstly, the limits converge a.s.\ as $\delta\downarrow 0$, because $\tau(r\pm\delta)\to \tau(r)$ and the process $\xi$ is continuous at $\tau(r)$ and at $\tau(r)+\tilde t_i$. 
Secondly, we must show that the excluded integrated difference is uniformly negligible:
\begin{equation*}
	\lim_{\delta\downarrow0}\limsup_{n\to\infty}\p(\sup_{t\leq \lceil\tau(r+\delta)n\rceil/n}|n\int_{[\tau(r-\delta)n]/n}^{[tn]/n}(\exp(\a\xi_s)-\exp(\a\xi^\n_s))\dd s|\geq\epsilon)=0.
\end{equation*}
But the respective quantity converges weakly according to Theorem~\ref{thm:jacod}, and the limit goes to 0 a.s.\ establishing this claim.
The proof is now complete.
\end{proof}

\subsection{Extension to multivariate case}
Let us recall a basic result, which readily follows from Skorokhod's representation theorem.
Assume that $\mu_n$ is a sequence of finite measures converging weakly to a finite measure $\mu$ and that $f_n$ is a sequences of bounded functions that are continuously convergent, i.e. $f_n(z_n)\to f(z)$ whenever $z_n\to z$ for $z$ in the support of $\mu$.
Then we also have
\[\int f_n\dd\mu_n\to \int f\dd\mu.\]

\begin{proof}[Proof of Theorem~\ref{thm:jointU}]
We prove the multivariate case inductively. Suppose the case $d\geq 1$ is proven.
Consider $r=(r_d+r_{d+1})/2$, and let $\tau_n=\lceil \tau(r)n\rceil/n$ be the epoch following $\tau(r)$ which is a stopping time. Note also that $\tau_n\to\tau(r)$ and  $r^\n=I_{\tau_n}\to r$ a.s.  We condition on $\xi_{\tau_n}=x$ and $r^\n-r=\epsilon$ and use the strong Markov property to split the quantities of interest.
The processes $\Delta^\n$ are split into two parts: the one stopped at $\tau_n$ and the post-$\tau_n$ contribution. The latter corresponds to $\exp(\a x)\tilde\Delta^{\n}$ for the process $\tilde \xi_t=\xi_{\tau_n+t}-\xi_{\tau_n}$ which is independent of~$\ff_{\tau_n}$.
Moreover, note that 
\[\tau(r_{d+1})=\tau_n+\tilde\tau(\exp(-\a x)(r_{d+1}-r-\epsilon))\]
and zooming in at $\tau(r_{d+1})$ translates into zooming in on $\tilde \xi$ at the respective time, whereas
\[\{\tau(r_{d+1})n\}=\{\tilde\tau(\exp(-\a x)(r_{d+1}-r-\epsilon))n\}.\]
Finally, we may assume that none of the zoomed-in processes over $[-T,T]$ span both $[0,\tau_n]$ and $(\tau_n,\infty)$ since this is true with arbitrarily large probability for large enough~$n$. This allows to split the variables of interest into two independent groups under the conditional law specified above.

Next, we construct the measures $\mu_n(\dd x,\dd \epsilon)$ and the functions $f_n$ by simply applying bounded continuous functions to the two quantities of interest, where the latter also include $\tilde\xi_{\tilde t_i}$ needed to guarantee the stable convergence.
Weak convergence of measures follows from the inductive assumption and the facts that $\tau_n\to\tau(r),r^\n\to r$ and $\xi$ is continuous at $\tau(r)$. 
Convergence of $f_n(x_n,\epsilon_n)$ for $(x_n,\epsilon_n)\to(x,0)$ follows from Lemma~\ref{lem:conv} with $r_n\to r$ given by
\[\exp(-\a x_n)(r_{d+1}-r-\epsilon_n)\to \exp(-\a x)(r_{d+1}-r).\] 
It is left to glue back the limits, where the only dependence comes from $x$ needed to reconstruct the process $(\Delta_t)_{\geq 0}$ and the variables $\xi_{\tau(r)+\tilde t_i}$.
Finally, note that convergence still holds when $\xi_{\tau_n+\tilde t_i}$ are replaced by $\xi_{\tau(r)+\tilde t_i}$ in the pre-limit. This yields the stated stable convergence for $d+1$, and the proof is complete. 
\end{proof}

\section{Zooming-in on pssMp}\label{sec:zooming}
\subsection{The result}
Self-similarity of $X$ implies that $n^{1/\a}X_{t/n}$ (with $X$ starting at $x$) has the law of $X_t$ (starting at $xn^{1/\a}$).
There is, however, a different scaling resulting in a limit process as $n\to\infty$, which we now state. 
Importantly, it provides a zooming-in limit for the pssMp $X$ and connects it to the zooming-in limit for~$\xi$.
It is noted that this result does not require the assumption~\eqref{eq:limsup}.
Furthermore, this result is somewhat related to the law of iterated logarithm for $X_t$ at small times, see~\cite[Thm.\ 7.1]{lamperti72} and~\cite[\S2.3]{rivero_survey}.
\begin{theorem}\label{thm:ppsMpzooming}
Under the assumption~\eqref{eq:zoominginlimit} there is the convergence for any $x>0$
\begin{equation}\label{eq:pssMpzooming}a_n(X_{t/n}-x)_{t\geq 0}\stably x^{1-\a/\beta}(\widehat\xi_t)_{t\geq 0},\qquad \R\ni n\to\infty,\end{equation}
where $\widehat\xi$ is independent of $\ff$ and $1/\beta$ is its Hurst index.

Furthermore,~\eqref{eq:zoominginlimit} is equivalent to the weak convergence of $a_n(X_{1/n}-1)$ to a non-zero limit for $x=1$.
\end{theorem}

%
\begin{proof}
For all $t\in[0,T]$ we have 
\[a_n(X_{t/n}-x)=xa_n\xi_{\tau(x^{-\a}t/n)}(1+\oh_\p(1)),\]
where the $\oh_\p(1)$ term depends on $T$ and not on~$t$.
It is left to show that
\begin{equation}\label{eq:neg_timechange}\sup_{t\leq T}|\tau(x^{-\a}t/n)n-x^{-\a}t|\cip 0,\end{equation}
since then by continuity of subordination~\cite[Thm.\ 13.2.2]{whitt} at the limiting time change $x^{-\a}t$, and $a_n\xi_{\cdot/n}\stably \widehat \xi$ we find 
\[(a_n\xi_{\tau(x^{-\a}t/n)})_{t\leq T}\stably (\widehat \xi_{x^{-\a}t})_{t\leq T}\stackrel{d}{=} x^{-\a/\beta}(\widehat \xi_{t})_{t\leq T}.\]
Finally,
\[t/n-\tau(t/n)=\int_0^{\tau(t/n)} (e^{\a\xi_s}-1)\dd s = \tau(t/n)\oh_\p(1),\]
which firstly shows that $n\tau(x^{-\a}T/n)$ is stochastically bounded and thus also establishes~\eqref{eq:neg_timechange}.

Next, assume that $a_n(X_{1/n}-1)\convd Z\neq 0$ for $x=1$. Then
\[a_n\xi_{\tau(1/n)}=a_n(e^{\xi_{\tau(1/n)}}-1)(1+\oh_\p(1))\convd Z.\]
But $\tau(1/n)n\cip 1$ and according to Proposition~\ref{prop:conv} below we must have \[a_n\xi_{1/n}\convd Z.\]
The proof is now complete.
\end{proof}

Let us note that the non-zero weak limit of $a_n(X_{1/n}-1)$, when it exists, is necessarily $\widehat\xi_1$.
In fact, this assumption is equivalent to a seemingly weaker assumption, namely that $a_n(X_{t/n}-x)\convd Z\neq 0$ for some $t,x>0$.
Importantly, Theorem~\ref{thm:ppsMpzooming} allows to identify $\widehat\xi$ directly without determining the corresponding process $\xi$ first.
An application of this will be given in  \S\ref{sec:application} below.

\subsection{On convergence of L\'evy processes at random times}
The following basic result is essential for the second statement in Theorem~\ref{thm:ppsMpzooming}, and it maybe useful in various other settings.
Somewhat surprisingly, it is not contained in the standard monographs.
\begin{proposition}\label{prop:conv}
Consider a sequence of L\'evy processes $\xi^n$ and assume that $\xi^n_{T_n}\convd Z$ for some random times $0\leq T_n\cip 1$.
Then $\xi^n_1\convd Z$.
\end{proposition}
Importantly, we do not assume that $\xi^n$ and $T_n$ are independent. 
The main difficulty is in proving that $\xi^n_1$ is tight, which is the content of the following two Lemmas.
\begin{lemma}\label{lem:levy_cip0_bounded}
Assume that $0\leq T_n\cip 1$ and $\xi^n_{T_n}\cip 0$ for a sequence of L\'evy processes $\xi^n$ such that
\begin{equation}\label{eq:assumption}\p(\sup_{t\leq 1}|\xi^n_t|>1)\leq 1/2.\end{equation}
Then $\xi^n_1\cip 0$.
\end{lemma}
\begin{proof}
Suppose for contradiction that there exist $\epsilon,\delta>0$ such that
\[\liminf_{n\to\infty}\p(\xi^n_{1-\delta}<-2\epsilon)>0.\]
Let $\xi'^n_t=\xi^n_{1-\delta+t}-\xi^n_{1-\delta}$ be the incremental post-$(1-\delta)$ process.
Then 
\[\p(\xi^n_{T_n}>-\epsilon,|T_n-1|<\delta,\xi^n_{1-\delta}<-2\epsilon)\leq \p(\xi^n_{1-\delta}<-2\epsilon)\p(\overline \xi'^n_{2\delta}>\epsilon),\]
where on the right-hand side we used independence of $\xi'$ and $\xi^n_{1-\delta}$. By the initial assumption we readily obtain 
\[p_n:=\p(\overline \xi^n_{2\delta}>\epsilon)\to 1\qquad\text{as }n\to\infty.\]
Applying strong Markov property at first passage times we now find
\[1/2\geq \p(\overline \xi^n_1>1)\geq p_n^{\lceil 1/\epsilon\rceil}\]
for all $n$, given that $2\delta\lceil 1/\epsilon\rceil<1$. 
In this case the right-hand side tends to $1$, which is a contradiction.
Similar reasoning works when $\p(\xi^n_{1-\delta}>2\epsilon)$ is assumed to be bounded away from~0.
Thus we conclude that for any $\epsilon>0$ and small enough $\delta>0$ we have
\[\p(|\xi^n_{1-\delta}|>\epsilon)\to 0\qquad\text{as }n\to\infty.\]

Fix arbitrary $h,\epsilon>0$ and choose $\delta$ small so that 
$\p(|\xi^n_{1-2\delta}|<\epsilon)$ and $\p(|\xi^n_{1-\delta}|<\epsilon)$ are larger than $1-h$ for all large $n$.
Thus $\p(|\xi^n_{\delta}|<2\epsilon)>1-2h$ implying that $\p(|\xi^n_1|<3\epsilon)>1-3h$, which completes the proof.
\end{proof}

\begin{lemma}\label{lem:levy_cip0}
The conclusion of Lemma~\ref{lem:levy_cip0_bounded} is true without the assumption~\eqref{eq:assumption}.
\end{lemma}
\begin{proof}
We choose the maximal $b_n$ such that~\eqref{eq:assumption} is satisfied for $\xi'^n_t=b_n\xi^n_t$:
\[b_n=\sup\{b\in(0,1]:\p(\sup_{t\leq 1}|b\xi^n_t|>1)\leq 1/2\}.\]
Since $b_n$ is upper bounded by construction, we still have $\xi'^n_{T_n}\cip 0$. Now the previous lemma implies that $b_n\xi^n_1\cip 0$,
and then according to the standard theory~\cite[Thm.\ 15.17]{kallenberg} we also have convergence on the process level. By the continuous mapping theorem we find
\[b_n\sup_{t\leq 1}|\xi^n_t|\cip 0,\]
whereas by maximality of $b_n$ it must be that $\p(\sup_{t\leq 1}|2b_n\xi^n_t|>1)> 1/2$ for any~$b_n<1$.
Hence $b_n=1$ for all large $n$ and the proof is complete.
\end{proof}

\begin{proof}[Proof of Proposition~\ref{prop:conv}]
Take any sequence $0\leq h_n\to 0$ and note that $h_n\xi^n_{T_n}\cip 0$. By Lemma~\ref{lem:levy_cip0} we also have $h_n\xi^n_1\cip 0$.
According to~\cite[Lem.\ 4.9]{kallenberg} the sequence $\xi^n_1$ is tight. It is standard~\cite[Prop.\ 5.21]{kallenberg} that every subsequence has a weakly convergent further subsequence $\xi^{n_k}_1$.
Then it must be~\cite[Thm.\ 5.12]{kallenberg} that the limit is $Z'_1$ for some L\'evy process~$Z'$,
and $\xi^{n_k}\convd Z'$, see~\cite[Thm.\ 15.17]{kallenberg}.
But $Z'$ is necessarily continuous at time~$1$ a.s., and thus $\xi^{n_k}_{T_{n_k}}\convd Z'(1)$ showing that $Z'_1$ and $Z$ have the same distribution.
Thus $\xi^{n_k}_1\convd Z$ and the proof is now complete.
\end{proof}

\section{Application to self-similar L\'evy processes conditioned to stay positive}\label{sec:application}
\subsection{Definition and properties}
Let $X^0$ be a non-monotone $1/\a$-self-similar L\'evy process. 
In particular, $X^0$ is either (I) a drift-less Brownian motion ($\a=2$) or (II) a strictly $\a$-stable L\'evy process with $\a\in(0,2)$. 
Without real loss of generality we may fix the scale, and so in case (I) we assume that $X_0$ is a standard Brownian motion.
We also define the negativity parameter
\[\rho=\p(X^0_1<0)\in(0,1),\]
which additionally must satisfy $\a-1\leq \a\rho\leq 1$ and, in particular, $\rho=1/2$ in case~(I).

Let $X$ be the process $X^0$ conditioned to stay positive when started from $x>0$. Formally it is defined via Doob's $h$-transform~\cite{lampertirepresentation}:
\begin{equation}\label{eq:htransform}
  	\pp_x^\uparrow(A):=h^{-1}(x)\ee[h(x+X^0_t)\ind{A}\ind{x+\underline X^0_t>0}],\qquad t\geq0,A\in\mathcal F_t,
\end{equation}
where $h(x)=x^{\a\rho}$, see also~\cite{bertoin_splitting} for the case when $X^0$ is a general L\'evy process.
We write $(X,\p)$ for the pair $(x+X^0,\p^\uparrow_x)$ and specify $x>0$ separately when needed. 
Let us also mention that the new law in~\eqref{eq:htransform} coincides with the limit of $\p(A|x+\underline X^0_s>0)$ as $s\to\infty$, explaining the name 'conditioned to stay positive'. 
Importantly, $X$ is a strictly positive pssMp with Hurst parameter~$1/\a$.
Such processes naturally arise in limit theory concerned with extremes, first passage times and Skorokhod reflection, see~\cite{ivanovs2018,iva_podolskij} and references therein.
Importantly, in case (I) the process $X$ is Bessel-3.

As before, let $\xi$ be the L\'evy process in the Lamperti representation~\eqref{eq:lamperti} of~$X$.
In case (I) the process $\xi$ is a Brownian motion with unit variance and drift $1/2$,~see~\cite{cpy}.
In case (II) the L\'evy triplet of $\xi$  has been identified in~\cite{lampertirepresentation}, excluding the non-symmetric Cauchy case.
It is worth mentioning that $\xi$ has no Brownian component, and its L\'evy density behaves as the L\'evy density of the original stable process $X^0$ both at~$0+$ and at~$0-$.
Furthermore, $\xi$ is a pure jump process when $\a\in(0,1)$.
Finally, we note that~\cite[Eq.\ (17)]{lampertirepresentation} has a typo: the second term should come with a minus sign, which only affects the drift parameter.

Importantly, Theorem~\ref{thm:ppsMpzooming} allows to identify $\widehat\xi$ and to verify assumption~\eqref{eq:zoominginlimit} without the knowledge of~$\xi$.
It turns out that $\widehat\xi$ has the law of the original process $X^0$ and, in particular, $\beta=\a$.
\begin{proposition}\label{prop:conditioned}
Let $X$ be $X^0$ conditioned to stay positive.
Then the assumption~\eqref{eq:zoominginlimit} is satisfied with
\[a_n=n^{1/\a},\qquad\widehat\xi\stackrel{d}{=}X^0.\]
\end{proposition}
\begin{proof}
According to Theorem~\ref{thm:ppsMpzooming} we only need to verify that \[n^{1/\a}(X_{1/n}-1)\convd X^0_1\] for $x=1$ as $\R\ni n\to\infty$.
Using \eqref{eq:htransform} and self-similarity of $X^0$ one easily verifies that
\begin{equation*}
	\Prob{n^{1/\a}(X_{1/n}-1)\leq z}=\Mean{(n^{-1/\a}X_1^0+1)^{\a\rho}\ind{X_1^0\leq z}\ind{n^{-1/\a}\underline X_1^0>-1}},
\end{equation*}
for any $z\in\R$. But the right-hand side converges to $\Prob{X_1^0\leq z}$ since $X_1^0$ has no atoms, and we are done.
\end{proof}

Alternatively, one may prove Proposition~\ref{prop:conditioned} using the knowledge of the L\'evy triplet of $\xi$ by checking the conditions of~\cite[Thm.\ 2]{ivanovs2018}.
The latter approach requires verification that the drift parameter of $\xi$ is zero in case $\a<1$, and this is how we found a typo in~\cite{lampertirepresentation}.
Furthermore, calculations are somewhat tedious in the symmetric Cauchy case, whereas the triplet of $\xi$ in the non-symmetric case is not yet available.

In various applications the law of interest corresponds to the weak limit of $\p_x^\uparrow$ as $x\downarrow 0$, which corresponds to the conditioned process started at~0.
This can be approximated by taking small $x>0$, which then results in large $r=x^{-\a}t$. This does not seem to be a problem due to the structure of $L(r)$ and the fact that $\xi_t\to\infty$ as $t\to\infty$.

\subsection{Simulations}

Here we present a small simulation study in order to illustrate our results. 
For simplicity, we take a standard Brownian motion conditioned to stay positive (Bessel-3 process) as the pssMp $X$ of interest.
Let us stress that simple and exact simulation methods exist for Bessel-3 process, and our only purpose is to illustrate the results of~\S\ref{sec:mainresults}.
In this case $\alpha=2,a_n=\sqrt n$, $\xi$ is a Brownian motion with unit variance and drift $1/2$, whereas $\widehat{\xi}$ is a standard Brownian motion, see Proposition~\ref{prop:conditioned}. 
We also note that assumption~\eqref{eq:density} is satisfied since the density $g_t(x,y)$ in Lemma~\ref{lem:density} is indeed jointly continuous in $t,x,y>0$, see~\cite[1.8.8, p.\ 613]{borodin}.

We start $X$ at $x=1$ and simulate at time $t=1$. Hence $X_1=\exp(\xi_{\tau(1)})$. We use two rather coarse discretization grids corresponding to $n=10$ and $n=100$.
The true quantities are computed using $N=10^6$, so that $\xi^{(N)}$ and $X^{(N)}$ are used in place of $\xi$ and $X$, respectively.
The process $\Delta$ in~\eqref{eq:Delta} is approximated by taking $\xi^{(N)}$ in the Brownian integral and removing the sum over jump times, which must be 0 in the case of continuous~$\xi$.
Finally, $\tau(1)$ is replaced by $\tau_N(1)$.
Importantly, the increments of $\xi^{(N)}$ are assembled into the increments of $\xi^{(n)}$, so that the two processes correspond to the same sample path.
These sample paths are then reused in construction of the limit variables.

In Figure \ref{fig:sim_results_time} below we compare the distributions of $n(\tau(1)-\tau_n(1))$ and the limit $L(1)$, see Proposition~\ref{prop:tau}. All histograms are based on simulation of 10,000 independent copies of the relevant random variable. 
In red we have $n(\tau(1)-\tau_n(1))$ and in blue we have $L(1)$. Since some values are quite large the histogram has been trimmed to contain at least $98\%$ of the realizations. More precisely the lower limit is the minimum of the $1\%$-quantiles for $n(\tau(1)-\tau_n(1))$ and $L(1)$, and the upper limit is the maximum of the $99\%$-quantiles.
We discuss these large values in detail later.
Let us remark that already at $n=10$ we see very similar histograms and at $n=100$ the fit is even better.
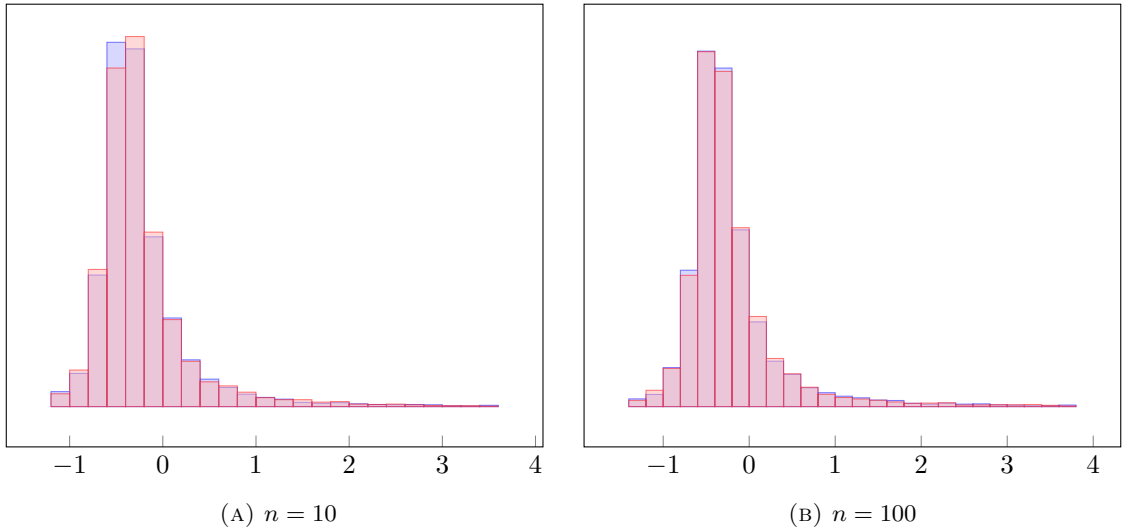
\begin{figure}[h]
	\centering
	\begin{subfigure}{.48\textwidth}
		\centering
		\begin{tikzpicture}
			\begin{axis}[area style, tick pos=left, ymajorticks=false, width=1.2\linewidth,ymin=-0.15,ymax=1.5]
				\addplot+[ybar interval, opacity=0.5] table[x="V1",y="V2", col sep=comma] {histogram_L.csv};
				\addplot+[ybar interval, opacity=0.5] table[x="V1",y="V2", col sep=comma] {histogram_prelimit_tau.csv};
			\end{axis}
		\end{tikzpicture}
		\caption{$n=10$}
	\end{subfigure}\quad%
	\begin{subfigure}{.48\textwidth}
		\centering
		\begin{tikzpicture}
			\begin{axis}[area style, tick pos=left, ymajorticks=false, width=1.2\linewidth,ymin=-0.15,ymax=1.5]
				\addplot+[ybar interval, opacity=0.5] table[x="V1",y="V2", col sep=comma] {histogram_L2.csv};
				\addplot+[ybar interval, opacity=0.5] table[x="V1",y="V2", col sep=comma] {histogram_prelimit_tau2.csv};
			\end{axis}
		\end{tikzpicture}
		\caption{$n=100$}
	\end{subfigure}
	\caption{Histograms for $n(\tau(1)-\tau_n(1))$ in red and $L(1)$ in blue trimmed to contain at least $98\%$ of the realizations.}
	\label{fig:sim_results_time}
\end{figure}

In Figure \ref{fig:sim_results_space} we depict the discretization errors for the pssMp itself. That is, we compare the distributions of $\sqrt{n}(X_1-X^{(n)}_1)/X_1$ in red and the limit $\widehat\xi_{L(1)+U}$ in blue, see Corollary~\ref{cor:mainFDD}. The fit is worse than in Figure~\ref{fig:sim_results_time}, which is to be expected since now we combine the error in time and the zooming-in approximation. Again we have trimmed the histograms to contain at least $98\%$ of the realizations.
\begin{figure}[h]
	\centering
	\begin{subfigure}{.48\textwidth}
		\centering
		\begin{tikzpicture}
			\begin{axis}[area style, tick pos=left, ymajorticks=false, width=1.2\linewidth,ymin=-0.1,ymax=1]
				\addplot+[ybar interval, opacity=0.5] table[x="V1",y="V2", col sep=comma] {histogram_xihat.csv};
				\addplot+[ybar interval, opacity=0.5] table[x="V1",y="V2", col sep=comma] {histogram_prelimit_X.csv};
			\end{axis}
		\end{tikzpicture}
		\caption{$n=10$}
	\end{subfigure}\quad%
	\begin{subfigure}{.48\textwidth}
		\centering
		\begin{tikzpicture}
			\begin{axis}[area style, tick pos=left, ymajorticks=false, width=1.2\linewidth,ymin=-0.1,ymax=1]
				\addplot+[ybar interval, opacity=0.5] table[x="V1",y="V2", col sep=comma] {histogram_xihat2.csv};
				\addplot+[ybar interval, opacity=0.5] table[x="V1",y="V2", col sep=comma] {histogram_prelimit_X2.csv};
			\end{axis}
		\end{tikzpicture}
		\caption{$n=100$}
	\end{subfigure}
	\caption{Histograms for $\sqrt{n}(X_1-X^{(n)}_1)/X_1$ in red and $\widehat\xi_{L(1)+U}$ in blue trimmed to contain at least $98\%$ of the realizations.}
	\label{fig:sim_results_space}
\end{figure}
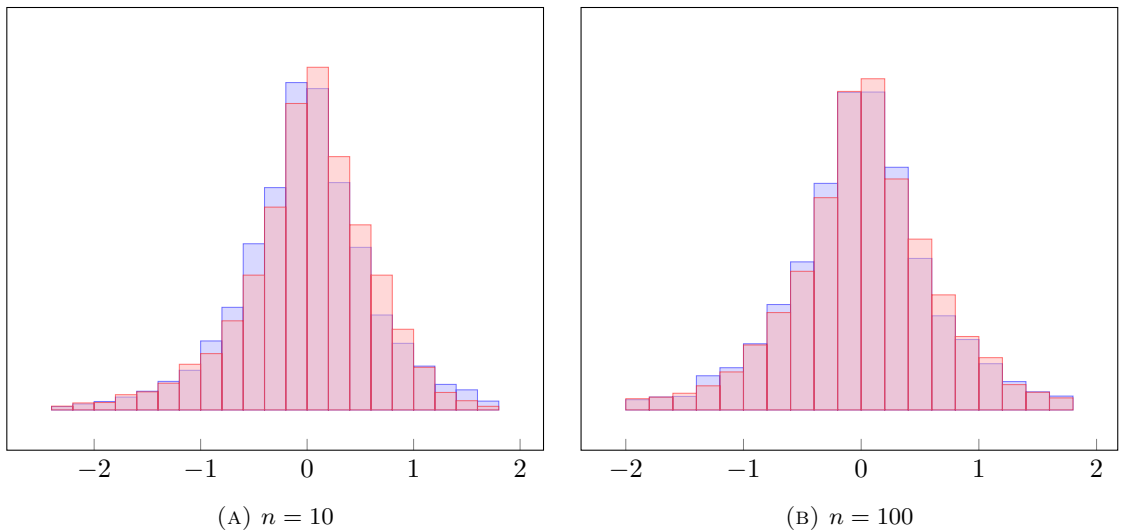

In order to understand the extreme values of $L(1)$ and $n(\tau(1)-\tau_n(1))$ we depict a sample path in Figure \ref{subfig:extreme_expxi} which results in large values of both variables. In this case $n=10$. Notice that $I^{(n)}$ hits $1$ right before $\exp(\a\xi_t)$ vanishes ($-\xi_t$ becomes large), whereas $I$ hits $1$ upon much later.
This illustrates how $n(\tau(1)-\tau_n(1))$ can become very large. Furthermore, $\exp(\a\xi_{\tau(1)})=X^\a_1$ is close to zero and so $L(1)=-\Delta_{\tau(1)}X^{-\a}_1$ can be large as well.
It seems that heaviness of the tails of $L(1)$ is determined by $X^{-\a}_1$; for the Bessel-3 process this quantity has a power tail with exponent $-3/2$, see~\cite[1.0.6, p.\ 373]{borodin}.
\begin{figure}[H]
	\centering
	\begin{subfigure}[t]{.48\textwidth}
		\vskip 0pt
		\centering
		\begin{tikzpicture}
			\begin{axis}[tick pos=left, width=1.15\linewidth, ymin=-0.3,ymax=3,ytick={0,1,2},xmin=-2,xmax=20,xtick={0,5,10,15}]
				\addplot+ [mark=none, color=blue] table[x="V1",y="V2", col sep=comma] {extreme_exponentiallevy.csv};
				\addplot+ [jump mark left,mark=none, color=red] table[x="V1",y="V2", col sep=comma] {extreme_discretizedexponentiallevy.csv};
			\end{axis}
		\end{tikzpicture}
		\captionsetup{width=.8\linewidth}
		\caption{The processes $\exp(\a\xi_t)$ and $\exp(\a\xi^\n_t)$ in respectively blue and red with $n=10$.}
		\label{subfig:extreme_expxi}
	\end{subfigure}\quad%
	\begin{subfigure}[t]{.48\textwidth}
		\vskip 0pt
		\centering
		\begin{tikzpicture}
			\begin{axis}[tick pos=left, width=1.15\linewidth,ymin=-0.13,ymax=1.3,ytick={0,1},xmin=-2,xmax=20,xtick={0,5,10,15}]
				\addplot+ [mark=none, color=blue, solid] table[x="V1",y="V2", col sep=comma] {extreme_integral.csv};
				\addplot+ [mark=none, color=red, solid] table[x="V1",y="V2", col sep=comma] {extreme_discretizedintegral.csv};
				\draw[dashed,>=stealth,color=blue](axis cs:18.0963,-0.13)--(axis cs:18.0963,1);
        		\draw[dashed,>=stealth,color=red](axis cs:1.23677,-0.13)--(axis cs:1.23677,1);
        		\draw[dashed,>=stealth,color=black](axis cs:-2,1)--(axis cs:20,1);
			\end{axis}
		\end{tikzpicture}
		\captionsetup{width=.8\linewidth}
		\caption{Integrals of the processes in Figure \ref{subfig:extreme_expxi}. The colored dashed lines mark $\tau(1)$ and $\tau_n(1)$.}
	\end{subfigure}
	\caption{A sample path and corresponding integrals producing extreme values of $L(1)$ and $n(\tau(1)-\tau_n(1))$.}
	\label{fig:extreme_prelimit_tau}
\end{figure}
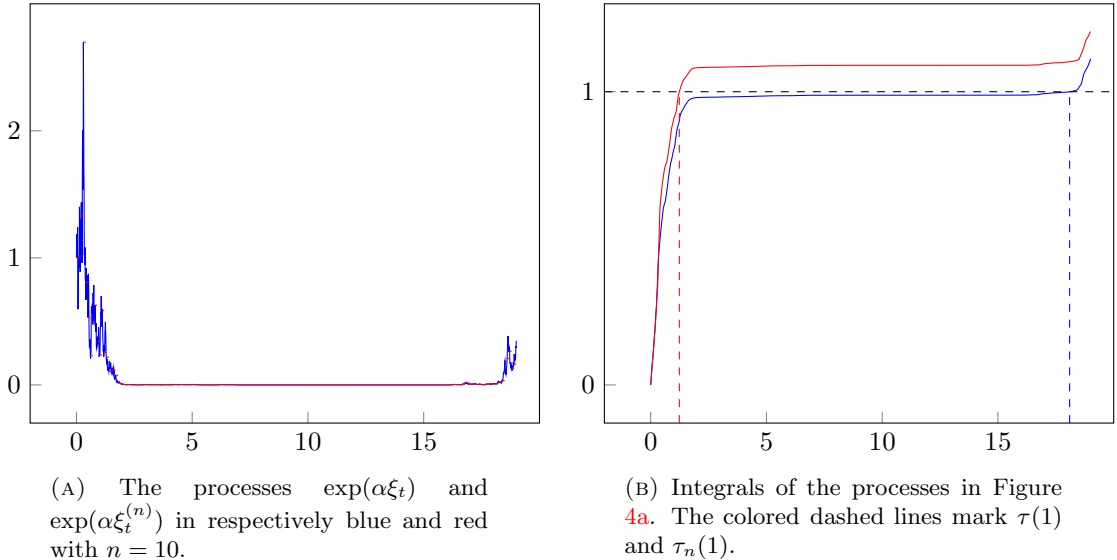

In conclusion, discretization provides the standard rate of convergence $n^{-1/\a}$, but the limit variables normally exhibit heavy-tails.

\section{Extensions and comments}\label{sec:comments}
\subsection{Trapezoidal approximation}\label{sec:trapezoid}
An interesting modification of our approximation scheme is obtained by considering the trapezoidal rule (instead of the left Riemann sum) in computation of the integral~$I_t$, so that the points $(i/n,\exp(\a\xi_{i/n}))$ are connected by straight lines. 
Importantly, all the results and proofs of this paper continue to hold true given that Theorem~\ref{thm:jacod} and the definition of $\Delta$ in~\eqref{eq:Delta} are adjusted accordingly, which we now discuss.

Observe that the trapezoidal approximation $\widetilde I^\n$ satisfies
\[\widetilde I^\n_{i/n}=I^\n_{i/n}+(f(\xi_{i/n})-f(0))/(2n)\]
and hence the form of the new limiting process is intuitively clear:
\[\widetilde\Delta_t=\frac{\sigma^2}{\sqrt{12}}\int_0^t f'(\xi_s)\dd W'_s+\sum_{m:T_m\leq t}(f(\xi_{T_m})-f(\xi_{T_m-}))(\kappa_m-\frac{1}{2}),\]
that is, the bias $(f(\xi_t)-f(0))/2$ is removed from~\eqref{eq:Delta}. 
\begin{theorem} The trapezoidal approximation $\tilde I^\n$ satisfies
\[n\left(I_{[tn]/n}-\widetilde I^{(n)}_{[tn]/n}\right)_{t\geq 0}\stably(\widetilde\Delta_t)_{t\geq 0}.\]
\end{theorem}
\begin{proof}
The proof requires only some simple adaptations of the proof in~\cite[Ch.\ 6]{jacodbook}. Firstly, we note that the reduction of the problem in \S6.2.2 is still true, because of the u.c.p. convergence of processes in (6.2.13). Secondly, (6.3.6) now contains our new term, which is rewritten using Itô's formula, and the limiting expression in (6.3.7) is modified accordingly. The expressions, in fact, become even shorter, and the rest of the proof applies.
\end{proof}

It must be noted that this result can not be directly retrieved from~\cite[Thm.\ 6.1.2]{jacodbook} and the basic relation between left Riemann sum and trapezoidal rule.
The problem is that the continuous mapping theorem does not apply for the sum of the two processes of interest since both components may jump at the same time. This issue does not arise in the setting of continuous Itô semimartingales considered in~\cite{altmeyer1}.

\subsection{Absolute continuity of the inverse}\label{sec:density}
Here we establish a sufficient condition for~\eqref{eq:density} in terms of the integral $I_t$ and the end-value $\exp(\a\xi_t)$. 
We assume that the pair
\[(I_t,Y_t)=\big(\int_0^t \exp(\a \xi_s)\dd s,\exp(\a \xi_t)\big)\] 
has a density $g_t(x,y)$ for all $t>0$.
Recall that $\tau(r)$ is defined by the relation $I_{\tau(r)}=r$, and that we need simple conditions implying that the pair $(\tau(r),Y_{\tau(r)})$ has a density for all~$r>0$. 

\begin{lemma}\label{lem:density}
	Assume that $g_t(x,y)$ is jointly continuous in $x,y,t>0$. Then for any $r>0$
	\[\p(\tau(r)\in\dd t,Y_{\tau(r)}\in\dd y)=yg_t(r,y)\dd t\dd y,\qquad t,y>0.\]
\end{lemma}	
\begin{proof}
	For fixed $r>0$ and $0<a<b<\infty$ consider 
	\[F(t)=\p(\tau(r)\leq t,Y_{\tau(r)}\in [a,b]),\qquad t\geq 0.\]
	We note that it is sufficient to show that $F(t)$ is a (left-) continuous function with the right-derivative 
	\begin{equation}\label{eq:partial}\partial_+F(t)=\int_a^b yg_t(r,y)\dd y=:f(t)\end{equation}
	for all $t>0$. This is so, because $f$ is continuous and so with $G(t)=\int_0^tf(u)\dd u$ we have
	$\partial_+(F(t)-G(t))=0$ for all $t>0$, implying that $F(t)$ coincides with $G(t)$ on $t>0$ up to a constant.
	By taking $t\downarrow 0$ we see that this constant is~0 and hence
	\[F(t)=\int_0^t\int_a^b yg_u(r,y)\dd y\dd u\]
	establishing the claim.
	
	For $h>0$ we note the identity
	\begin{equation}\label{eq:Fdif}F(t+h)-F(t)=\p(I_t<r\leq I_{t+h},Y_{\tau(r)}\in [a,b]).\end{equation}
	Moreover, $I_{t+h}=I_t+Y_tI_h'$ with $I_h'$ corresponding to $\xi_u'=\xi_{t+u}-\xi_t$, and so the latter is independent of $\mathcal F_t$. 
	Next, we note for any $z>0$ that 
	\begin{align}\label{eq:lim_dens}\frac{1}{h}\p(I_t<r\leq I_t+Y_tzh,Y_{t}\in [a,b])=\frac{1}{h}\int_a^b\int_{r-yzh}^r g_t(x,y)\dd x \dd y
	\to z\int_a^byg_t(r,y)\dd y.\end{align}
	This follows from the mean value theorem and the fact that $g_t(x,y)$ is bounded for $x$ of interest.

	Define  $\Delta'_h=\exp(\a\sup_{u\leq h} \xi'_u)$ and note that $I'_h\leq h\Delta'_h$.
	Moreover, for any $\epsilon>0$ we may choose $c>1$ large enough so that 
	$\p(\Delta'_h>c)<\epsilon h$ for $h$ small enough. This can be seen from the inequality~\cite[Lem.\ 2, p.\ 420]{gikhman_skorokhod}
	\[\p(\sup_{u\leq h} |\xi'_u|>\log c/\a)\leq (1+\oh(1))\p(|\xi'_h|\geq \log c/(2\a)),\]
	and the standard bounds on the right-hand side, see the argument in~\cite[Lem.\ 30.3]{sato}. 
	Thus in the following we may always assume that $\Delta'_h\leq c$.
	Similarly, we may also assume that $\underline\Delta'_h=\exp(\a\inf_{u\leq h} \xi'_u)\geq \underline c\in(0,1)$.
	
	Now, we readily find that 
	\[\p(I_t<r\leq I_t+Y_tch,Y_t\in[a/c,b/\underline c],\Delta'_h>1+\epsilon)=\oh(h)\]
	and the analogous statement with $\underline\Delta'_h<1-\epsilon$.
	Hence we have the following upper bound on \eqref{eq:Fdif}
	\[\p(I_t<r\leq I_t+Y_tI'_h,I'_h<ch,Y_t\in [a/(1+\epsilon),b/(1-\epsilon)])\]
	up to some negligible terms,
	and a similar lower bound. It is left to condition on $I'_h/h$, to apply the arguments from~\eqref{eq:lim_dens} and to notice that 
	\[\lim_{h\downarrow 0}\e(I_h/h\ind{I_h/h<c})=1,\]
	where the latter is a consequence of the mean value theorem and the dominated convergence theorem.
	Hence \eqref{eq:partial} is now proven.
	Left-continuity of $F(t)$ follows from $\p(I_{t-h}<r\leq I_{t-h}+(b/\underline c)ch)\to 0$.  
\end{proof}

\section*{Acknowledgements}
The authors gratefully acknowledge financial support of Sapere Aude Starting Grant 8049-00021B ``Distributional Robustness in Assessment of Extreme Risk'' from Independent Research Fund Denmark.


\end{document}